\documentclass[a4paper, 10pt, twoside]{article}
\usepackage[utf8]{inputenc}
\usepackage{graphicx}
\usepackage{hyperref}
\usepackage{amsmath}
\usepackage{amsthm}
\usepackage[english]{babel}
\usepackage{amsfonts}
\usepackage{enumerate}
\usepackage{amssymb}
\usepackage{color}

\usepackage{mathrsfs}
\usepackage{latexsym}
\usepackage{mhequ}
\usepackage{verbatim}
\usepackage{xcolor}
\usepackage{tikz}
\usetikzlibrary{calc,intersections,through,backgrounds}
\usetikzlibrary{decorations.pathreplacing}
\usetikzlibrary{shapes}

\usepackage{float}
\usepackage[margin=1in, marginparwidth=.8in]{geometry}
\usepackage{setspace}

\usepackage{fancyhdr}
\usepackage{url}

\usepackage[]{appendix}

\newcommand{\abs}[1]{\left|#1\right|}
\newcommand{\eps}{\varepsilon}

\renewcommand{\d}{\mathrm{d}}
\renewcommand{\P}{\mathbb{P}}
\newcommand{\Plr}[1]{\P\left( #1\right)}
\newcommand{\E}{\mathbb{E}}
\newcommand{\Ind}[1]{\mathds{1}_{#1}}

\newcommand{\Om}{\Omega}
\newcommand{\vphi}{\varphi}

\newcommand{\mc}[1]{\mathcal{#1}}

\newcommand{\cB}{\mc{B}}

\newcommand{\cO}{\mc{O}}

\newcommand{\norm}[1]{\left\lVert #1\right\rVert}

\renewcommand{\sp}[2]{\left\langle #1,#2 \right\rangle}

\newcommand{\Hsob}{{H^1_0}}
\newcommand{\Hd}{{H^{-1}}}

\newcommand{\dup}[4]{{}_{#1}{\left\langle #2, #3 \right\rangle}_{#4}}

\newcommand{\N}{\mathbb{N}}

\newcommand{\R}{\mathbb{R}}

\newcommand{\ie}{i.\,e.\ }
\newcommand{\eg}{e.\,g.\ }

\newtheorem{thm}{Theorem}[section]
\newtheorem{prop}[thm]{Proposition}
\newtheorem{lem}[thm]{Lemma}
\newtheorem{Def}[thm]{Definition}
\newtheorem{Cor}[thm]{Corollary}

\theoremstyle{remark}
\newtheorem{Rem}[thm]{Remark}

\makeatletter
\DeclareRobustCommand\widecheck[1]{{\mathpalette\@widecheck{#1}}}
\def\@widecheck#1#2{%
    \setbox\z@\hbox{\m@th$#1#2$}%
    \setbox\tw@\hbox{\m@th$#1%
       \widehat{%
          \vrule\@width\z@\@height\ht\z@
          \vrule\@height\z@\@width\wd\z@}$}%
    \dp\tw@-\ht\z@
    \@tempdima\ht\z@ \advance\@tempdima2\ht\tw@ \divide\@tempdima\thr@@
    \setbox\tw@\hbox{%
       \raise\@tempdima\hbox{\scalebox{1}[-1]{\lower\@tempdima\box
\tw@}}}%
    {\ooalign{\box\tw@ \cr \box\z@}}}
  \makeatother

\renewcommand{\Ind}[1]{\mathrm{\textbf{1}}_{#1}}

\newcommand{\tix}{{\tilde{x}}}
\newcommand{\tiy}{\tilde{y}}
\newcommand{\g}{g}
\newcommand{\ome}{{\omega_1}}
\newcommand{\omz}{{\omega_2}}

\newcommand{\xeps}{{x,\eps}}

\title{Ergodicity for singular-degenerate porous media equations}
\author{Marius Neuß}
\date{}

\fancypagestyle{plain}{

  \fancyhf{}
  \fancyfoot[L]{\footnotesize \today}
  \fancyfoot[C]{\thepage}
  }

\numberwithin{equation}{section}

\begin{document}
\maketitle

\begin{abstract}
  The long time behaviour of solutions to generalised
  stochastic porous media equations on bounded domains with Dirichlet
  boundary data is studied. We focus on a degenerate form of nonlinearity
  arising in self-organised criticality. Based on the
  so-called lower-bound method, the existence and uniqueness of an invariant
  measure is proved.


\end{abstract}

\section{Introduction}
We consider the singular-degenerate generalised stochastic porous medium equation
\begin{align}\label{eq:33}
  \begin{split}
    \d X_{t}&\in\ \Delta(\phi(X_{t}))\d t+B \d W_{t},\\
    X_0 &=\  x_0,
  \end{split}
\end{align}
on a bounded interval $\cO\subseteq \R$ with zero Dirichlet boundary
conditions. The multi-valued function $\phi$ is the maximal monotone
extension of
\begin{equation}\label{eq:49}
  \R\ni x\mapsto x \Ind{\{\abs{x} > 1\}},
\end{equation}
$W$ is a cylindrical Wiener process on some separable Hilbert space $U$,
and the diffusion coefficient $B$ is an $L^2(\cO)$-valued Hilbert-Schmidt
operator satisfying a non-degeneracy condition (see
\eqref{eq:21} below). Equation \eqref{eq:33} is understood as an evolution
equation on $\Hd$, the dual of $\Hsob(\cO)$, where it can be solved uniquely in
the sense of SVI solutions, as shown in \cite{Neuss}. The main result of
the present work is the existence and uniqueness of an invariant
probability measure for solutions to \eqref{eq:33}.

The above form of stochastic porous media equations is motivated by the
analysis of non-equilibrium systems, appearing in the context of
self-organised criticality (for a survey, see \eg
\cite{Watkins-Pruessner}). Self-organised criticality is a statistical
property of systems displaying intermittent events, such as earthquakes,
which are activated when the underlying system locally exceeds a
threshold. These dynamics are reflected by the discontinuity and degeneracy of the
nonlinearity $\phi$ above. In order to get a better understanding of the long-time
behaviour of these systems, we prove the existence of a unique
non-equilibrium statistical invariant state for \eqref{eq:33}. Since this
is the candidate to which the transition probabilities are expected to
converge for long times, it is the key object for the statistical behaviour
of the respective process.

A previous approach to the long-time behaviour of Markov processes stemming
from monotone SPDEs with singular drift, by which the present article is
inspired, is \cite{G-T-ergodicity}, which in turn uses the more abstract
framework of \cite{G-T-multivalued}. In these works, the existence and
uniqueness of invariant probability measures to stochastic local and
non-local $p$-Laplace equations is proved, where the multivalued regime
$p=1$ is included. In one dimension, the paradigmatic case is the equation
\begin{equation}\label{eq:48}
  \d X_t = \Delta (\mathrm{sgn}(X_t)) + \d W_t,
\end{equation}
where $\mathrm{sgn}$ denotes the maximal monotone extension of the classical
sign function. The proof relies on
sufficient criteria from \cite{Komorowski}, where the so-called lower-bound
technique has been extended to Polish spaces which are not necessarily
locally compact. This technique relies on the existence of a state being an
accessible point for the time averages of the transition probabilities
uniformly in time, and the so-called ``e-property'', which is a uniform
continuity assumption on the Markov semigroup. To verify these criteria,
the focus of \cite{G-T-ergodicity} rests on energy estimates to first bound
the mass of these averages to $L^m$ balls for some suitably chosen
$m\in (2,3]$. As a next step, the convergence to a chosen
accessible state with probability bounded below is shown, which is done by
comparing the solution of \eqref{eq:48} to a control process, which obeys
the mere deterministic dynamics of \eqref{eq:48}, \ie
\begin{align}\label{eq:50}
  \begin{split}
    \frac{\d}{\d t} X_t &= \Delta (\mathrm{sgn}(X_t)),\\
    X_0 &= y,
  \end{split}
\end{align}
for $y\in L^m, \norm{y}_m \leq R$ for some $R>0$. In this, simpler setting
than \eqref{eq:33}, there is a unique limiting state to \eqref{eq:50} which
is a natural candidate for the aforementioned accessible point.

In the present article, we aim to prove the existence and uniqueness of an
invariant probability measure by similar ideas. While energy estimates for
\eqref{eq:33} are easier to obtain due to the linear growth of $\phi$
(cf.\,\eqref{eq:49}) at $\pm \infty$, the degenerate form of the
nonlinearity destroys the convergence of the noise-free system to a unique
fixed point. This is why we have to add a forcing term to the control
process and rely on a more refined deterministic analysis of the
resulting inhomogeneous monotone evolution equation. To guarantee the
convergence of this modified control process, the forcing term has to be
sufficiently non-degenerate, and as the connection of the solution to
\eqref{eq:33} to the control process only works if the noise is ``close''
to the deterministic forcing with non-zero probability, this relies on
some non-degeneracy requirements on the
noise. As in\cite{G-T-ergodicity}, it is important
that the convergence of the deterministic process takes place uniformly for
initial values in sets of bounded energy. We tackle this problem with
the help of a comparison principle, which, however, only works if the
energy actually controls the $L^\infty$ norm. This leads to the restriction
to one spatial dimension. Finally, most of the above-mentioned steps
have to be argued on an approximate level due to the singularity of the
drift, so that stability of the statements under these approximations also
has to be ensured.

\subsection{Literature}
The well-posedness of SPDEs with monotone, multivalued drift has been
investigated by \cite{BDPrR-existence-strong} and
\cite{BDPrR-SPMEandSOC}. The concept of stochastic variational inequalities
(SVIs) and a corresponding notion of solution has been established in
\cite{BDPrR-Diffusivity} and \cite{B-R-SVITVF}, and has been applied to
generalised stochastic $p$-Laplace equations in \cite{G-R-plaplace} and to
generalised stochastic fast diffusion equations in \cite{G-R}. Finally, the existence
and uniqueness of SVI solutions to \eqref{eq:33} follows from a more
general well-posedness analysis in \cite{Neuss}.

We now aim to give a brief overview on the existing results on ergodicity
of stochastic nonlinear diffusions, with a focus on approaches applicable to
stochastic (generalised) porous media equations.

In the ``classical'' approach, \eg in the monograph
\cite{DPr-Zab-Ergodicity}, the existence of invariant measures to
semilinear SPDEs with non-degenerate noise is proven by bounds that imply
the tightness of the averaged transition probabilities, allowing to use the
Krylov-Bogoliubov theorem. Uniqueness is then relying on the
Doob-Khasminskii theorem, using the regularity of the Markov semigroup
which can be guaranteed by the strong Feller property and
irreducibility. This technique has been considerably improved by
\cite{Hairer-Matt-NS}, using smoothing in form of the asymptotic Feller
property, though the scope was still on semilinear equations.

Invariant measures to quasilinear diffusions with additive noise have been
initially studied in \cite{DPrR-Invariant} and \cite{DPrR-weaksolns} on the
level of Kolmogorov equations. In \cite{DPrRRW} (see also the monograph
\cite{BDPrR-SPMEbook}), the strong monotonicity of the porous medium
operator was exploited, which leads to the existence and uniqueness of
invariant measures by strong dissipativity.

In the situation of weakly monotone drift operators, there have been
several approaches to obtain contraction estimates which ensure
ergodicity, \eg via Harnack inequalities (cf.\,\cite{Wang2015, Wang2007}),
weighted $L^1$ dissipativity (cf.\,\cite{G-D-T}) or lower-bound techniques
(cf.\,\cite{Lasota-Szarek}, \cite{Komorowski}). We note that the first
approach also works for a partly multiplicative noise and the second one
even for full multiplicative noise. The last-mentioned approach was used by
\cite{G-T-multivalued} and \cite{G-T-ergodicity}, where generalised porous media
equations with discontinuous nonlinearities are analysed as explained
above.

A different approach to the long-time behaviour of solutions to SPDEs is to
analyse the existence and the structure of random attractors of random
dynamical systems, as \eg in
\cite{Crauel-Flandoli,CDF,Gess-RAdegenerate,Gess-Liu-Roeckner,BGLR,Gess-Liu-Schenke}. A
property which has turned out to be very useful in this context is
order preservation of trajectories which are driven by the same noise, see,
e.\,g., \cite{Gess-RAsingular,Arnold-Chueshov,FGS17-orderpres,Butkovsky-S}. A close connection between random attractors and ergodic and mixing
properties of random dynamical systems can be obtained in the case of
synchronisation (see \cite{Cranston-Gess-Scheutzow}), which is on hand if the random
attractor is a singleton. This case has been investigated in, e.\,g., \cite{Chueskov-Scheutzow,FGS17,FGS17-orderpres,Rosati,Butkovsky-S}.

Last but not least, we mention \cite{Barbu-Cellular, BDPrR-FTE, Gess-FTE},
where similar equations are considered under multiplicative noise, leading
to finite-time absorption of the process into a subcritical region.

\subsection{Structure of the paper}
After stating the exact setting in the first part of section
\ref{sec:setting}, we state the main result of this article, Theorem
\ref{main-result} at the end of section \ref{sec:setting}. Section
\ref{sec:lemmas-proof} then collects auxiliary results in the natural order
of the argumentation, which finally allow to prove Theorem
\ref{main-result}.

\subsection{Notation}
On a bounded open set $\cO\subset \R$, we use the classical notations
$L^p:=L^p(\cO)$ for the Lebesgue space with exponent $p\in[1,\infty]$ and
$\Hsob := \Hsob(\cO)$ for the Sobolev space of weakly differentiable
functions with exponent $2$ and zero trace. The norm on $L^p$ will be
denoted by $\norm{\cdot}_p$. A bounded operator $T: U\to H$,
where $U$ and $H$ are separable Hilbert spaces, is called Hilbert-Schmidt if
\begin{displaymath}
  \norm{T}_{L_2(U,H)} := \sum_{k\in\N} \norm{T e_k}_H^2 < \infty,
\end{displaymath}
where $(e_k)_{k\in\N}$ is an orthonormal basis of $U$. For a Hilbert space
$H$, $C_b(H)$ denotes the space of bounded continuous functions, $\cB(H)$
denotes the Borel $\sigma$-algebra, and $\cB_b(H)$ the set of bounded
functions $H\to\R$ which are $\cB(H)$-$\cB(\R)$-measurable. Multivalued
operators on $H$, which arise in this work as subdifferentials of proper,
convex and lower-semicontinuous functionals, are mappings
$A: H \to 2^H$. We define the domain of $A$ by
\begin{displaymath}
  D(A) := \left\{x\in H: A(x) \neq \emptyset\right\}
\end{displaymath}
and its range by
\begin{displaymath}
  R(A) := \bigcup_{x\in H} A(x).
\end{displaymath}
For a metric space
$V$ and $r>0$, we denote by $B_r^V$ the open ball with radius $r$ with
respect to the corresponding metric. If $V=L^\infty$, we use $B_r^\infty$
for $B_r^{L^\infty}$. Within term manipulations, the constant $C$ may vary
from line to line.

\section{Setting and main result} \label{sec:setting}
We consider a one-dimensional open bounded interval $\cO \subset \R$ as the
underlying domain. For simplicity, set $\cO := (-1,1)$. 

Define by $\phi: \R\to 2^\R$ the multi-valued maximally monotone extension of
\begin{displaymath}
  \R\ni x \mapsto x \Ind{\{\abs{x}>1\}},
\end{displaymath}
and let $\psi: \R\to\R$ be its anti-derivative with $\psi(0) = 0$, \ie
\begin{displaymath}
  \psi(x) = \frac1{2}(\abs{x}^2 -1)\Ind{\{\abs{x}>1\}}.
\end{displaymath} 
Let furthermore $\vphi: \Hd\to [0,\infty]$ be defined as
\begin{equation}\label{eq:14}
  \vphi(u) =
  \begin{cases}
    \int_\cO \psi(u)\,\d x, \quad &\text{if }u\in L^2\\
    +\infty,\quad &\text{else.}
  \end{cases}
\end{equation}
and consider the SPDE
\begin{align}\label{eq:1}
  \begin{split}
    \d X^x_t &\in -\partial \vphi(X^x_t)\,\d t + B\d W_t,\\
    X^x_0 &= x,
  \end{split}
\end{align}
where $x\in\Hd$, $W$ is an $\mathrm{Id}$-cylindrical Wiener process on some separable Hilbert
space $U$ and $B\in L_2(U, L^2)$ a Hilbert-Schmidt
operator. This leads to $B W_t$ being a trace-class Wiener process in
$L^2$, such that there are mutually orthogonal $L^2$ functions
$(\xi_k)_{k\in\N}$ with
\begin{equation}\label{eq:51}
\sum_{k\in\N}\norm{\xi_k}^2_2 < \infty,
\end{equation}
for which
\begin{equation} \label{eq:25}
  B W_t = \sum_{i=1}^\infty \beta_k(t) \xi_k,
\end{equation}
where $(\beta_k)_{k\in\N}$ are independent one-dimensional standard
Brownian motions. Additionally, we impose that there are $m\in\N$, $c_1, \dots,
c_m\in\R$ such that
\begin{equation}\label{eq:21}
  \g\in L^2, \quad\g(x) := \sum_{k=1}^m c_k \xi_k(x) > 1\quad\text{for almost all }x\in\cO.
\end{equation}
Note that the well-posedness of the SPDE \eqref{eq:1} has been shown in
\cite{Neuss} in the sense of SVI-solutions, identifying $x$ with an
almost surely constant random variable $x\in L^2(\Om, \Hd)$. The process constructed there
gives rise to a semigroup $(P_t)_{t\geq 0}$ of Markov transition kernels by
\begin{equation}\label{eq:59}
    P_t (x, A) = \E \Ind{A}(X_t^x)\quad \text{for }x\in\Hd \text{ and }A\in
    \cB(\Hd),
  \end{equation}
  which will be shown below in Lemma \ref{Markov-prop}. By a slight abuse
  of notation, we will denote the induced semigroup on $\cB_b(\Hd)$ also by
  $P_t$, \ie
  \begin{equation}\label{eq:60}
    P_tf (x) = \int_\Hd f(y) P_t(x, \d y) \quad \text{for }f\in \cB_b(\Hd),\,
    x\in \Hd.
  \end{equation}

The main result of this article is the following:
\begin{thm}\label{main-result}
  In the setting described above, the semigroup $(P_t)_{t\geq 0}$ admits a
  unique invariant probability Borel measure $\mu$ on $\Hd$, \ie for all
  $f\in C_b(\Hd)$ we have
  \begin{displaymath}
    \int_\Hd P_tf \d \mu = \int_\Hd f \d \mu.
  \end{displaymath}
\end{thm}

We briefly mention the steps of the proof. After we introduce the main
approximating object $X^\xeps$ to solutions $X^x$ of \eqref{eq:1}, we prove
a contraction principle, \ie
\begin{displaymath}
  \P\left(\norm{X_T^x - X_T^y}_\Hd \leq \norm{x-y}_\Hd\right) = 1\quad
  \text{for all }T>0,
\end{displaymath}
which will be needed throughout the remaining proof. The lower-bound technique of \cite{Komorowski} is then applied in three steps: We first prove
that solutions to \eqref{eq:1} are likely to stay on average close to a ball in
$L^\infty$, \ie for $\rho, \delta > 0$ there exists an $R>0$ such that for
sufficiently large $T>0$
\begin{equation}\label{eq:30}
  \frac1{T}\int_0^T \P(X^x_r \in C_\delta(R))\,\d r \geq 1-\rho,
\end{equation}
where $C_\delta(R)$ is the $\delta$-neighbourhood of $B_R^\infty(0)$ in
$\Hd$. We then analyse the deterministic equation
\begin{align*}
  \begin{split}
    \frac{\d}{\d t} u^{\pm R} &= - \partial \vphi(u^{\pm R}) + \g,\\
    u^{\pm R}(0) &\equiv \pm R,
  \end{split}
\end{align*}
which will serve as the control process mentioned above and which converges
for large times to a limit $u_\infty\in\Hd$. Finally, we show that with positive
probability, $X^x$ behaves ``similar'' to $u^{\pm R}$ if $x\in
C_\delta(R)$, so that together with \eqref{eq:30} we can conclude that for
all $x\in \Hd$, $\delta>0$
\begin{displaymath}
  \liminf_{T\to\infty} \frac1{T} \int_0^T
  P_r\left(x,B_{2\delta}^\Hd(u_\infty)\right)\d r > 0,
\end{displaymath}
which implies the existence and uniqueness of an invariant measure by
\cite[Theorem 1]{Komorowski}.

\section{Lemmas and proof}\label{sec:lemmas-proof}
We recall the following notion from \cite{Komorowski}:
\begin{Def}\label{e-prop-Def}
  We say that a transition semigroup $(P_t)_{t\geq 0}$ on some Hilbert
  space $H$ has the e-property
  if the family of functions $(P_tf)_{t\geq 0}$ is equicontinuous at
  every point $x\in H$ for any bounded and Lipschitz continuous function
  $f: H\to\R$.
\end{Def}

As mentioned before, the proof of the main theorem relies on the following
sufficient condition of \cite{Komorowski}:

\begin{prop}[Komorowski-Peszat-Szarek 2010] \label{KPS} Let
  $(P_t)_{t\geq 0}$ be the transition semigroup of a stochastically
  continuous Markov process taking values on a separable Hilbert space $H$. Assume
  that $(P_t)_{t\geq 0}$ satisfies the Feller- and the
  e-property. Furthermore, assume that there exists $z\in H$ such that for
  every $\delta > 0$ and $x\in H$
  \begin{equation}\label{eq:56}
    \liminf_{T\to\infty} \frac1{T}\int_0^T P_r(x, B^H_\delta(z))\d r > 0.
  \end{equation}
  Then the semigroup $(P_t)_{t\geq 0}$ admits a unique, invariant probability Borel
  measure.
\end{prop}

Most of the following arguments involve an approximating process, which
will be introduced in the following lemmas.

\begin{lem}\label{recall-Yosida}
  Let $\phi^\eps$ denote the Yosida approximation of $\phi$, as introduced
  in \cite[Appendix D]{Neuss}. Then $\phi^\eps$ is Lipschitz continuous,
  monotonically increasing and piecewise affine.
\end{lem}

\begin{proof}
  This is clear by the results in \cite[Appendix D]{Neuss}.
\end{proof}

\begin{lem}\label{recall-approx}
  Let $T>0$ and $x\in L^2$, and consider
  the SPDE
  \begin{align}\label{eq:34}
    \begin{split}
      \d X_t^{x,\eps} &= \eps \Delta X_t^{x,\eps}\d t + \Delta
      \phi^\eps(X_t^{x,\eps})\d t + B\d W_t,\\
      X_0 &= x.
    \end{split}
  \end{align}
  Then, identifying $x$ with an almost surely constant random variable
  $x \in L^2(\Om, L^2)$, \eqref{eq:34} allows for a unique variational
  solution $(X_t^{x,\eps})_{t\in[0,T]}$ in the sense of \cite[Definition
  4.2.1]{Roeckner} with respect to the Gelfand triple
  $\Hsob \hookrightarrow L^2 \hookrightarrow\Hd$. $X^{x,\eps}$ satisfies
  the regularity estimate
  \begin{equation}\label{eq:38}
    \E \sup_{t\in[0,T]} \norm{X^\xeps_t}_2^2 + \eps \E  \int_0^T
    \norm{X^\xeps_r}^2_{H^1_0}\d r \leq C(T)(\E\norm{x}_2^2 + 1)
  \end{equation}
  with a constant $C(T)>0$ independent of $\eps$. Furthermore, for
  $(x_n)_{n\in\N}\subset L^2$, $x_n \to x$ in $\Hd$ for $n\to\infty$, we have
  \begin{equation}\label{eq:35}
    \lim_{n\to\infty} \lim_{\eps\to 0} X^{x_n,\eps} = X^x,
  \end{equation}
  where $X^x$ is the SVI solution to \eqref{eq:1} and the limits are taken
  in $L^2\left(\Om, C([0,T], \Hd)\right)$. More precisely, the $\eps$-limit is uniform
  on bounded sets of $L^2$ by the estimate
  \begin{equation}\label{eq:47}
    \E \sup_{t\in[0,T]} \norm{X^{y,\eps} - X^{y}}^2_\Hd \leq \eps\, C(T)(\norm{y}_2^2+1)
  \end{equation}
  for $y\in L^2$, and for the $n$-limit we have
  \begin{equation}\label{eq:61}
    \E\sup_{t\in[0,T]}\norm{X^{x_n} - X^x}^2_\Hd \leq C(T)\norm{x - x_n}_\Hd^2.
  \end{equation}
  Finally, for $x,y\in\Hd$ we have
  \begin{equation}\label{eq:37}
    \sup_{t\in[0,T]}\E\norm{X_t^x - X_t^y}_\Hd^2 \leq C \norm{x - y}_\Hd^2.
  \end{equation}
\end{lem}

\begin{proof}
  This becomes clear by \cite[sections 5.1 -- 5.3]{Neuss}, where the
  well-posedness for \eqref{eq:34} goes back to \cite[Theorem 4.2.4]{Roeckner}, always
  identifying $x$ and $y$ with almost surely constant random variables in
  the respective spaces. For the quantitative estimates, see especially
  \cite[equation (5.6)]{Neuss}. 
\end{proof}

\begin{Rem}\label{consistency}
  We note that if $0<T_1 < T_2<\infty$, $x\in L^2$, $X^\xeps$ is a solution to \eqref{eq:34}
  constructed on $[0,T_1]$ and $Y^\xeps$ is a solution to \eqref{eq:34}
  constructed on $[0,T_2]$, then $(Y^\xeps_t)_{t\in[0,T_1]}$ is also a
  solution to \eqref{eq:34}. By the uniqueness part of \cite[Theorem
  4.2.4]{Roeckner}, we have
  \begin{displaymath}
    X^\xeps_t = Y^\xeps_t \quad \text{for all }t\in [0,T_1].
  \end{displaymath}
  Consequently, $X^\xeps_t$ is consistently defined for all $t\geq 0$,
  $x\in\Hd$, and
  the same is true for $X^x_t$ by \eqref{eq:35}.
\end{Rem}

From \cite[section 4.3]{Liu-R}, we recall the following disintegration result.
\begin{lem}\label{Markov-lem-approx}
  The solution to \eqref{eq:34} is a time-homogeneous Markov process, such
  that we have
  \begin{displaymath}
    \E f(X_{t+s}^{x,\eps}) = \E_{\ome} \E_{\omz}
    f(X_t^{X_s^{x,\eps}(\ome), \eps}(\omz))
  \end{displaymath}
  for any bounded, $\cB(L^2)$-measurable $f: L^2\to\R$ and $t,s >0$.
\end{lem}

We need that solutions to \eqref{eq:1} are almost surely contractive,
which will be important in the subsequent analysis.
\begin{lem}\label{Contractivity-prop}
  Let $x,y\in\Hd$ and let $(X_t^x)_{t\geq 0}$ and $(X_t^y)_{t\geq 0}$ be
  the SVI solutions to \eqref{eq:1} with initial value $x$ and $y$,
  respectively. Then for all $T>0$ we have
  \begin{equation}\label{eq:67}
    \P\left(\norm{X_T^x - X_T^y}_\Hd \leq \norm{x-y}_\Hd\right) = 1.
  \end{equation}
\end{lem}

\begin{proof}
  We first fix $T>0$ for which we want to show the statement.

  \textit{Step 1:} First we prove contractivity on the level of approximate
  solutions and $x,y\in L^2$. For this, let $(X_t^{x,\eps})_{t\in[0,T]}$
  and $(X_t^{y,\eps})_{t\in[0,T]}$ solve \eqref{eq:34} with the respective
  initial value. Let furthermore $Z_t := X_t^{x,\eps} - X_t^{y,\eps}$,
  which solves
  \begin{align*}
    \d Z_t &= \eps \Delta (X_t^{x,\eps} - X_t^{y,\eps})\,\d t + \left(\Delta
             \phi^\eps(X_t^{x,\eps}) - \Delta\phi^\eps(X_t^{y,\eps})\right)\d t,\\
    Z_0 &= x-y.
  \end{align*}
  Then, by Ito's formula (see \eg \cite[Theorem 4.2.5]{Roeckner}), and
  noting that $Z\in\Hsob$ $\P\otimes\d t$-almost surely by \eqref{eq:38},
  we obtain $\P$-almost surely
  \begin{align*}
    \norm{Z_t}^2_\Hd
    &= \norm{x-y}^2_\Hd + 2\eps\int_0^t \sp{\Delta Z_r}{Z_r}_\Hd\d r
                 + \int_0^t\sp{\Delta\phi^\eps(X_r^{x,\eps}) -
                     \Delta\phi^\eps(X_r^{y,\eps})}{Z_r}_\Hd\d r\\
    &= \norm{x-y}^2_\Hd - 2\eps\int_0^t\norm{Z_r}_2\d r
      - \int_0^t\sp{\phi^\eps(X_r^{x,\eps}) - \phi^\eps(X_r^{y,\eps})}{X_r^{x,\eps} - X_r^{y,\eps}}_{L^2}\d r.
  \end{align*}
  The last two terms (the latter because of the monotonicity of
  $\phi^\eps$) are negative, which yields
  \begin{equation}\label{eq:66}
    \P\left(\norm{X_T^{x,\eps} - X_T^{y,\eps}}_\Hd - \norm{x-y}_\Hd >
      0\right) = 0.
  \end{equation}

  \textit{Step 2:} We now turn to SVI solutions for $x,y\in L^2$. Note
  that it is enough to show for arbitrary $n\in\N, \gamma > 0$ that
  \begin{equation}\label{eq:10}
    \P\left(\norm{X_T^x - X_T^y}_\Hd - \norm{x-y}_\Hd > \frac1{n}\right) \leq \gamma.
  \end{equation}
  To obtain this, choose $\eps$ sufficiently small such that by \eqref{eq:47}
  \begin{displaymath}
    \max\left\{\E\norm{X_T^{x,\eps}-X_T^x}_\Hd ,
      \E\norm{X_T^{y,\eps}-X_T^y}_\Hd\right\} < \frac{\gamma}{4n},
  \end{displaymath}
  which yields by Markov's inequality that
  \begin{displaymath}
    \P\left(\norm{X_T^{x,\eps}-X_T^x}_\Hd \geq \frac1{2n}\right) \leq
    \frac{\gamma}{2}
  \end{displaymath}
  and the corresponding statement for $X_T^y$. Thus together with
  \eqref{eq:66} we have 
  \begin{align*}
    \P&\left(\norm{X_T^x - X_T^y}_\Hd - \norm{x-y}_\Hd >
      \frac1{n}\right)\\
 \leq &\;\P\left(\norm{X_T^x - X_T^{x,\eps}}_\Hd \geq \frac1{2n}\right)
      + \P\left(\norm{X_T^y - X_T^{y,\eps}}_\Hd \geq \frac1{2n}\right)\\
      &+ \P\left(\norm{X_T^{x,\eps} - X_T^{y,\eps}}_\Hd - \norm{x-y}_\Hd
        > 0\right)\\
    \leq &\; \gamma,
  \end{align*}
  which yields \eqref{eq:67} in the case $x,y\in L^2$.
 
  \textit{Step 3:} Finally consider $x,y\in\Hd$. By \eqref{eq:37} we know
  that for $x,y\in\Hd$
  \begin{displaymath}
    \E\norm{X_T^x - X_T^y}_\Hd \leq C \norm{x-y}_\Hd.
  \end{displaymath}
  In order to confirm \eqref{eq:10}, we choose $\tix, \tiy\in L^2$ in a way
  that ($\norm{\cdot} = \norm{\cdot}_\Hd$)
  \begin{displaymath}
    \max\left\{\norm{x-\tix}, \norm{y-\tiy}\right\} \leq
    \frac1{4n}\quad \text{and}\quad \max\left\{C\norm{x-\tix},
      C\norm{y-\tiy}\right\} \leq \frac{\gamma}{8n}.
  \end{displaymath}
  Using
  \begin{displaymath}
    \norm{x-y} = \norm{x-\tix + \tix - \tiy + \tiy - y}
    \geq \norm{\tix - \tiy} - \norm{x-\tix} - \norm{y - \tiy}
  \end{displaymath}
  and, again by Markov's inequality,
  \begin{displaymath}
    \max\left\{\Plr{\norm{X_T^x - X_T^{\tix}} \geq \frac1{4n}}, \Plr{\norm{X_T^y - X_T^{\tiy}} \geq \frac1{4n}}\right\}  \leq \frac{\gamma}{2},
  \end{displaymath}
  we can compute
  \begin{align*}
    \P&\left(\norm{X_T^x - X_T^y}_\Hd - \norm{x-y}_\Hd > \frac1{n}
    \right)\\
    \leq &\; \Plr{\norm{X_T^x - X_T^{\tix}} \geq \frac1{4n}}
    + \Plr{\norm{X_T^{\tix} - X_T^{\tiy}} -
           \norm{\tix - \tiy} > 0}
           + \Plr{\norm{X_T^{\tiy} - X_T^y} \geq \frac1{4n}}\\
    &\; + \Plr{\norm{x -\tix} \geq \frac1{4n}} + \Plr{\norm{y-\tiy} \geq
      \frac1{4n}}\\
    \leq &\; \gamma,
  \end{align*}
  which finishes the proof.
\end{proof}

\begin{lem}\label{Markov-prop}
  The solution to \eqref{eq:1} gives rise to a semigroup of Markov transition kernels by
  \begin{displaymath}
    P_t (x, A) = \E \Ind{A}(X_t^x)\quad \text{for }x\in\Hd \text{ and }A\in
    \cB(\Hd).
  \end{displaymath}
  The induced semigroup $(P_t)_{t\geq 0}$ on $\cB_b(\Hd)$, given by
  \begin{displaymath}
    P_tf (x) = \int_\Hd f(y) P_t(x, \d y).
  \end{displaymath}
  has the Feller- and the e-property. For all $x\in\Hd$ and $f\in C_b(\Hd)$,
  \begin{equation}\label{eq:57}
    [0,\infty) \ni t\mapsto P_tf(x)
  \end{equation}
  is continuous at $t=0$.
\end{lem}

\begin{Rem}\label{existence-canonical-process}
  The semigroup $(P_t)_{t\geq 0}$ consisting of Markov transition kernels
  together with the obvious fact
  \begin{displaymath}
    P_0(x, A) = \Ind{A}(x)
  \end{displaymath}
  implies that there is a ``canonical'' Markov process with transition
  probabilities $(P_t)_{t\geq 0}$ (see \eg \cite[section
  2.2]{DPr-Zab-Ergodicity}).
\end{Rem}

\begin{Rem}\label{stoch-continuity}
  Note that the last statement in Lemma \ref{Markov-prop} implies the
  stochastic continuity of $(P_t)_{t\geq 0}$ by \cite[Proposition
  2.1.1]{DPr-Zab-Ergodicity}. By \cite[Theorem 2.2.2]{DPr-Zab-Ergodicity},
  the corresponding canonical process is then also stochastically continuous.
\end{Rem}

\begin{proof}[Proof of Lemma \ref{Markov-prop}:]
  The continuity of \eqref{eq:57} follows from the construction as an almost surely
  continuous process, and the Feller property from the contractivity in Lemma
  \ref{Contractivity-prop}. In both cases, we use that almost sure
  convergence implies convergence in probability, which in turn yields
  convergence in distribution by the Slutsky theorem (see \eg \cite[Theorem
  13.18]{Klenke}).

  To prove the e-property for $(P_t)_{t\geq 0}$, it is sufficient to show
  that for $f: \Hd\to\R$ bounded and Lipschitz continuous,
  $P_tf$ $(t\geq 0)$ is Lipschitz continuous with Lipschitz constant
  independent of $t$ and equal to the Lipschitz constant
  $[f]_{\mathrm{Lip}}$ of $f$. Using Lemma \ref{Contractivity-prop}, we can
  compute for $x,y\in\Hd$
  \begin{align*}
    \abs{P_tf(x) - P_tf(y)}
    &= \abs{\E\left[ f(X_t^x) - f(X_t^y)\right]}\\
    &\leq \E \abs{f(X_t^x) - f(X_t^y)}\\
    &\leq \E \left[ [f]_{\mathrm{Lip}} \norm{X_t^x - X_t^y}_\Hd\right]\\
    &\leq [f]_{\mathrm{Lip}} \norm{x-y}_\Hd,
  \end{align*}
  as required. 

  We turn to the kernel properties of $P_t$: For $x\in\Hd$, $t\geq 0$, $P_t(x, \cdot)$
  is the pushforward measure of $X^x_t$ and thereby a probability
  measure. Moreover, let $A\in \cB(\Hd)$. Note that the class of all functions $f\in
  \cB_b(\Hd)$, for which
  \begin{equation}\label{eq:54}
    \Hd \ni x\mapsto P_tf(x)
  \end{equation}
  is measurable, is monotone in the sense of \cite[Theorem 0.2.2, i) and ii)]{Revuz-Yor}. As the family of bounded Lipschitz functions generates
  the Borel $\sigma$-algebra and is stable under pointwise multiplication,
  \begin{displaymath}
    \Hd \ni x\mapsto P_t\Ind{A} (x)
  \end{displaymath}
  is proven to be measurable by the monotone class theorem (see \eg
  \cite[Theorem 0.2.2]{Revuz-Yor}), as soon as we show measurability of \eqref{eq:54} for
  bounded and Lipschitz continuous $f$. The latter, however, becomes clear
  by taking into account that $P_t f$ is Lipschitz continuous if $f$ is
  Lipschitz continuous (see the proof of the e-property above).

  To establish the semigroup property, we first note that the class of
  functions $f\in\cB_b(\Hd)$, for which the semigroup property
  \begin{equation}\label{eq:55}
    P_{t+s}f(x) = P_s(P_tf)(x)\quad \text{for all }t,s\geq 0,\, x\in\Hd
  \end{equation}
  holds, is also monotone, so that it is enough to prove the semigroup
  property for $f: \Hd\to\R$ being bounded and Lipschitz continuous. For
  such $f$, the claim follows by using the semigroup property for the
  approximating process $(X_t^{x_n, \eps})_{t\geq 0}$ with $\eps >0$,
  $n\in\N$, $(x_n)_{n\in\N} L^2$, $x_n\to x$ for $n\to\infty$ as stated in
  Lemma \ref{Markov-lem-approx}, and passing to the limit via Lemmas
  \ref{recall-approx} and \ref{Contractivity-prop}.
\end{proof}

The following lemma is an energy estimate for the $L^\infty$ norm.
\begin{lem} \label{lem-energy-estimate}
  Let $x\in \Hd$, $\delta, \rho > 0$ and for $R>0$
  \begin{displaymath}
    C_\delta(R) := \left\{ u\in\Hd: \exists v\in B_R^\infty(0) \text{ such
        that }\norm{u-v}_\Hd < \delta\right\},
  \end{displaymath}
  where $B_R^\infty(0) := \{ v\in L^\infty: \norm{v}_\infty<R\}$. Then there exists
  $R = R(\rho)>3$ such that for all $T>1$ we have
  \begin{equation}
    \frac1{T}\int_0^T \P(X^x_r \in C_\delta(R))\,\d r \geq 1-\rho.
  \end{equation}
  for solutions $X^x$ to \eqref{eq:1}.
\end{lem}
\begin{proof}
  We first consider the approximating solutions from \eqref{eq:34} with
  initial value $\tilde{x}\in L^2$, for which we know by \eqref{eq:38} that
  they are in $\Hsob$, $\P \otimes \d t$-almost surely. We choose $\tix$ in
  a way that
  \begin{equation}\label{eq:64}
    \norm{x - \tilde{x}}_\Hd \leq \frac{\delta}{2}.
  \end{equation}
  Note also that $\phi^\eps$ is weakly
  differentiable for $\eps>0$ and
  \begin{equation}\label{eq:2}
    (\phi^\eps)' \geq \frac1{2}\Ind{\R\setminus[-1,1]}
  \end{equation}
  for $0<\eps<1$ by \eqref{eq:62}. Ito's formula (see \eg \cite[Theorem
  4.2.5]{Roeckner}) on the Gelfand triple
  $\Hsob \hookrightarrow L^2 \hookrightarrow \Hd$ then yields
  \begin{align*}
    \norm{X_t^{\tix,\eps}}_2^2
    =& \norm{\tix}_2^2 + \int_0^t 2 \dup{\Hsob}{X_r^{\tix,\eps}}{\Delta(\eps X_r^{\tix, \eps} + \phi^\eps(X_r^{\tix, \eps}))}{\Hd}\,\d  r\\
     &+ \int_0^t 2\sp{X_r^{\tix, \eps}}{B\,\d W_r}_{L^2} + \int_0^t
       2\norm{B}_{L_2(U,L^2)}^2\,\d r.
  \end{align*}
  Abbreviating the last two summands by $K$ and using the chain rule for Sobolev functions (see \eg \cite[Theorem
  2.1.11]{Ziemer}) and \eqref{eq:2}, we obtain
  \begin{align}\label{eq:3}
    \begin{split}
      \norm{X_t^{\tix, \eps}}_2^2 &= \norm{\tilde{x}}_2^2 -2\eps\int_0^t \norm{\nabla
        X_r^{\tix, \eps}}_2^2\d r- \int_0^t\int_\cO 2\sp{\nabla
        X_r^{\tix, \eps}}{\nabla\phi(X_r^{\tix, \eps})}\,\d x\,\d r
      + K\\
      &\leq \norm{\tilde{x}}_2^2 - 2\int_0^t\int_\cO \phi'(X_r^{\tix, \eps})(\nabla
      X_r^{\tix, \eps})^2\,\d x\,\d  r + K\\
      &\leq \norm{\tilde{x}}_2^2 - \int_0^t\int_\cO
      \Ind{\left\{\abs{X_r^{\tix, \eps}}>1\right\}}(\nabla X_r^{\tix, \eps})^2\,\d
      x\,\d  r + K\\
      &= \norm{\tilde{x}}_2^2 - \int_0^t\int_\cO
      \left(\Ind{\left\{\abs{X_r^{\tix, \eps}}>1\right\}}\nabla X_r^{\tix, \eps}\right)^2\,\d x\,\d r
      + K.
    \end{split}
  \end{align}
  Defining $A\in \mathrm{Lip}(\R)$ by
  \begin{displaymath}
    x\mapsto A(x) = \mathrm{sgn}(x)\left(\abs{x} - 1\right)
    \Ind{\{\abs{x}>1\}},
  \end{displaymath}
  we see that almost everywhere
  \begin{displaymath}
    A'(X_r^{\tix, \eps}) = \Ind{\left\{\abs{X_r^{\tix, \eps}}>1\right\}}.
  \end{displaymath}
  Thus, using the chain rule for Sobolev functions and
   the continuous embedding $\Hsob \hookrightarrow L^\infty$, we can continue \eqref{eq:3} by
  \begin{align}\label{eq:6}
    \begin{split}
      \norm{X_t^{\tix, \eps}}_2^2 &\leq \norm{\tilde{x}}_2^2 - \int_0^t\int_\cO
      \left(\nabla A(X_r^{\tix, \eps})\right)^2\,\d x\,\d r + K\\
      &\leq \norm{\tilde{x}}_2^2 - C\int_0^t\norm{A(X_r^{\tix, \eps})}_\infty^2\,\d r + K\\
      &= \norm{\tilde{x}}_2^2 - C\int_0^t\left(\norm{X_r^{\tix, \eps}}_\infty -
        1\right)_+^2\,\d r + K.
    \end{split}
  \end{align}

  For the remaining part
  \begin{displaymath}
    K = \int_0^t 2\sp{X_r^{\tix, \eps}}{B\,\d W_r}_{L^2} + \int_0^t
       2\norm{B}_{L_2(U,L^2)}^2\,\d r
  \end{displaymath}
  we notice that the first summand vanishes in expectation and that the
  second one can be estimated from above by $Ct$ by the assumptions on $B$.
  Thus, taking expectations in \eqref{eq:6} provides
  \begin{equation}\label{eq:4}
    \E \int_0^t(\norm{X_r^{\tix, \eps}}_\infty-1)_+^2\,\d r \leq C\left(\norm{{\tilde{x}}}_2^2 + t\right),
  \end{equation}
  where we emphasize that $C$ does not depend on $\eps$.
  By the Markov inequality, we then use \eqref{eq:4} to compute
  \begin{align*}
    \frac1{T}\,\int_0^T\P\left((\norm{X_r^{\tix, \eps}}_\infty-1)_+^2 > R\right)\,\d r
    &\leq \frac1{T} \int_0^T \frac{\E\left(\norm{X_r^{\tix, \eps}}_\infty-1\right)_+^2}{R}\,\d r\\
    &\leq \frac{C}{TR}\left(\norm{{\tilde{x}}}_2^2 + T\right),
  \end{align*}
  which for $T>1$ becomes smaller than $\frac{\rho}{2}$ by choosing $R$
  large enough, uniformly in $\eps$. For technical reasons, we impose $R>3$
  without loss of generality. For $T>1$ fixed, we now choose
  $\eps$ small enough such that
  \begin{equation}\label{eq:9}
    \E \sup_{t\in [0,T]} \norm{X^\tix_t - X_t^{\tix, \eps}}_\Hd \leq \frac{\rho\delta}{4}.
  \end{equation}
  By Markov's inequality, \eqref{eq:9} yields
  \begin{displaymath}
    \P \left(\sup_{t\in [0,T]} \norm{X^\tix_t - X_t^{\tix, \eps}}_\Hd\geq \frac{\delta}{2}\right) \leq \frac{\rho}{2}.
  \end{displaymath}
  By Lemma \ref{Contractivity-prop} and \eqref{eq:64} we have for $t>0$
  \begin{displaymath}
    \norm{X^x_t - X^\tix_t}_\Hd \leq \frac{\delta}{2}\quad \text{almost surely,}
  \end{displaymath}
  which we use to conclude for $R$ as chosen above
  \begin{align}\label{eq:19}
    \begin{split}
      \frac1{T} \int_0^T \P(X^x_r \in C_\delta(R))\,\d r
      \geq& \frac1{T}\int_0^T \P(X^\tix_r\in C_{\frac{\delta}{2}}(R)\,\d r\\
      =& 1 - \frac1{T}\int_0^T \P(X^\tix_r \notin C_{\frac{\delta}{2}}(R))\,\d r\\
      \geq& 1 - \frac1{T}\int_0^T \P\left(\norm{X^\tix_r - X_r^{\tix, \eps}}_\Hd
        \geq {\frac{\delta}{2}}
        \text{ or } \norm{X_r^{\tix, \eps}}_\infty\geq R\right)\,\d r\\
      \geq& 1 - \frac1{T}\int_0^T \P\left(\norm{X^\tix_r - X_r^{\tix, \eps}}_\Hd
        \geq {\frac{\delta}{2}}\right)\\
      &\qquad \qquad\quad+ \P\left(\norm{X_r^{\tix, \eps}}_\infty \geq \sqrt{R} + 1\right)\, \d r\\
      \geq& 1 - \frac{\rho}{2} - \frac1{T}\int_0^T \P\left(
        (\norm{X_r^{\tix, \eps}}_\infty-1)_+^2 \geq R\right)\,\d r\\
      \geq& 1 - \rho,
    \end{split}
  \end{align}
  as required.
\end{proof}

We continue with the analysis of the deterministic control process, for
which we cite a translated version of \cite[Théorème 3.11]{Brezis}. For the
definition of weak and strong solutions, see Definition \ref{Def-det-solns}.
\begin{prop}\label{det-conv}
  Let $H$ be a Hilbert space and $A: H\supseteq D(A) \to H$ a maximal
  monotone operator of the form $A= \partial \vphi$ for some $\vphi: H\to
  [0,\infty]$ convex, proper and lower-semicontinuous. Suppose that for all $C\in\R$ the set
  \begin{equation}\label{eq:12}
    M:=\{x\in H: \vphi(x) + \norm{x}^2 \leq C\}
  \end{equation}
  is strongly compact. Let $f \in L^1_{\mathrm{loc}}([0,\infty); H)$ such
  that $\lim_{t\to\infty}f(t) =: f_\infty$ exists,
  $f-f_\infty \in L^1([0,\infty); H)$ and $f_\infty \in
  R(\partial\vphi)$. For $x \in \overline{D(\partial\vphi)}$, let $u^x$ be
  a weak solution to
  \begin{align*}
    \frac{\d}{\d t} u^x &\in - \partial\vphi(u^x) + f,\\
    u(0) &= x .
  \end{align*}
  Then $\lim_{t\to\infty}u^x(t) =: u_\infty$ exists and
  \begin{equation}\label{eq:63}
    f_\infty \in \partial\vphi(u_\infty).
  \end{equation}
\end{prop}

\begin{Rem}\label{Well-posedn-Brez}
  Note that existence even of strong solutions to \eqref{eq:13} is
  guaranteed by \cite[Théorèmes 3.4 and 3.6]{Brezis} for $t\in [0,T]$,
  $T>0$. By uniqueness, we can extend the solution to $[0,\infty)$,
  analogous to Remark \ref{consistency}. In particular, for $t>0$ and
  $x\in \overline{D(\partial\vphi)}$ we have $u^x(t) \in D(\partial\vphi)$.
\end{Rem}

From the definition of $\g$ in \eqref{eq:21}, recall especially that
$g\in L^2$ and $g>1$ almost everywhere in
$\cO$. For $x \in \overline{D(\partial\vphi)}$, consider the deterministic evolution
equation
\begin{align}\label{eq:13}
  \begin{split}
    \frac{\d}{\d t} u^x &\in - \partial \vphi(u^x) + \g,\\
    u^x(0) &= x
  \end{split}
\end{align}
on $\Hd$, where $\varphi$ is defined as in \eqref{eq:14}.

\begin{lem} \label{det-conv-applied}
  Let $R>1$. For the initial states $x\equiv \pm
  R$, Proposition \ref{det-conv} can be
  applied to problem \eqref{eq:13} by replacing both $f(t)$ and $f_\infty$
  by $\g$. In this case,
  \begin{equation}\label{eq:52}
    u_\infty = ((-\Delta)^{-1}\g) \vee 1.
  \end{equation}
\end{lem}
\begin{proof}
  The functional $\vphi$ as defined in \eqref{eq:14} is obviously not
  constantly $\infty$. Furthermore, it is convex and lower-semicontinuous
  by \cite[Proposition 2.10]{Barbu}.

  In order to verify the compactness of the set $M$ defined in
  \eqref{eq:12}, we first show for $\tilde{C}\in\R$ that $M$ is a bounded subset of
  $L^2$. This is obvious for $\tilde{C}\leq 0$ such that we can restrict to $\tilde{C}>0$
  in the following. Indeed, if for $u\in\Hd$ $\vphi(u)\leq \tilde{C}<\infty$, then $u\in
  L^2$ by \eqref{eq:14}. Then, we can compute
  \begin{align*}
    \int_\cO u^2 \d x
    &\leq \abs{\cO} + \int_{\{\abs{u}\geq 1\}} (\abs{u} -1+1)^2\d x\\
    &\leq \abs{\cO} + \int_{\{\abs{u}\geq 1\}} (\abs{u}-1)^2 + 2(\abs{u} - 1)
      + 1\, \d
      x\\
    &\leq \abs{\cO} + \vphi(u) + 2\abs{\cO}^{\frac1{2}}\left(\int_{\{\abs{u}\geq 1\}}
      (\abs{u}-1)^2\d x\right)^{\frac1{2}} + \abs{\cO}\\
    &\leq 2\abs{\cO} + \vphi(u) +
      2\abs{\cO}^{\frac1{2}}\vphi^{\frac1{2}}(u) \leq C\,(1 + \tilde{C}) < \infty.
  \end{align*}
  Since the canonical embedding $L^2\hookrightarrow \Hd$ is compact, it
  follows that $\overline{M}$ is compact. As $\vphi$ is
  lower-semicontinuous, so is $\vphi + \norm{\cdot}_\Hd^2$, and thus $M$ is
  also closed. Hence, $M$ is compact, as required.

  We recall from
  \cite[Proposition 2.10]{Barbu} that $\partial\vphi$ can be characterised by
  \begin{displaymath}
    \partial\vphi = \left\{ [u,w]\in (\Hd\cap L^1)\times\Hd: w = -\Delta v,
      v\in\Hsob, v(x) \in \phi(u(x)) \text{ for a.\,e.\,}x\in\cO\right\},
  \end{displaymath}
  with
  \begin{displaymath}
    D(\partial\vphi) = \left\{u\in \Hd\cap L^1: \exists\, v\in\Hsob \text{ such
        that } v \in\phi(u) \text{ almost everywhere}\right\}.
  \end{displaymath}
  To show that the constant functions $\pm R$ are elements of
  $\overline{D(\partial\vphi)}$, we define for $n\in\N$
  \begin{displaymath}
    v_n := n(1-x) \wedge n(x+1) \wedge R \in \Hsob,
  \end{displaymath}
  and $u_n := v_n \vee 1$. We then have $u_n\in \Hd\cap L^1$ and $v_n \in
  \phi(u_n)$, and thus $u_n \in D(\partial \vphi)$. Since $u_n \to R$ in $\Hd$, we have that the constant function
  $R\in \overline{D(\partial\vphi)}$. For the constant function with value $-R$,
  analogous considerations apply.
  
  Finally, to show \eqref{eq:52}, we first prove that
  \begin{equation}\label{eq:65}
    u_\infty = ((-\Delta)^{-1}\g)\vee 1
  \end{equation}
  satisfies \eqref{eq:63} with
  $f_\infty$ replaced by $\g$.
  Setting $v := (-\Delta)^{-1}\g$, we have $v\in\Hsob$, as $\g$ was assumed
  to be in $L^2\subset\Hd$, and consequently $v\vee 1\in \Hd\cap
  L^1$. Furthermore, $v>0$ almost everywhere by the strong maximum principle (see
  \cite[Theorem 8.19]{Gilb-Trud}) and thus $v \in \phi(v\vee 1)$ a.\,e., such
  that $v\vee 1 \in D(\partial\vphi)$. Since additionally $g= -\Delta v$,
  we have $\g\in R(\partial\vphi)$ and $g\in \partial\vphi(v\vee 1)$.

  We conclude by noticing that \eqref{eq:65} is the only choice for
  $u_\infty$ such that \eqref{eq:52} is satisfied. This becomes clear by
  the strict monotonicity of $\phi|_{\R\setminus (-1,1)}$ and the strict
  positivity of $(-\Delta)^{-1}\g$ by the strong maximum principle.
\end{proof}

Similarly to Lemma \ref{recall-approx}, we can define approximations
$u^\xeps$ for equation \eqref{eq:13} by
\begin{align}\label{eq:16}
  \begin{split}
    \frac{\d}{\d t} u^\xeps_t&= \eps \Delta u^\xeps_t + \Delta
    \phi^\eps(u^\xeps_t) + \g\quad \text{for }t\in (0,S], \\
    u^\xeps_0 &= x,
  \end{split}
\end{align}
where $S>0$ and $\g$ still satisfies assumption \eqref{eq:21}. Analogous to
the approximation of $X^x$, there is a unique variational solution to \eqref{eq:16}, and if
$x\in\overline{D(\partial\vphi)}\cap L^2$, so that \eqref{eq:13} has a
strong solution, we obtain
\begin{equation}\label{eq:58}
  \sup_{t\in[0,S]} \norm{u^\xeps_t - u^x_t}_\Hd^2 \leq C(S)(\norm{x}_2^2 + 1)
\end{equation}
analogous to \eqref{eq:47}.

For these approximating deterministic equations, we need order-preservation
in the initial value. A partial order on $\Hd$ can be defined as follows:
\begin{Def}\label{Def-Hd-order}
  We write $u\leq v$ in $\Hd$, if for all
  $\eta\in\Hsob, \eta\geq 0$ almost everywhere, one has
  \begin{displaymath}
    u(\eta) \leq v(\eta).
  \end{displaymath}
\end{Def}

\begin{lem}\label{sandwich-lemma}
  Let $u,v,w\in\Hd$. Then $u\leq v\leq w$ in $\Hd$ implies
  \begin{displaymath}
    \norm{v}_\Hd \leq \norm{u}_\Hd + \norm{w}_\Hd.
  \end{displaymath}
\end{lem}

\begin{proof}
  For arbitrary $\eta\in\Hsob, \norm{\eta}_\Hsob \leq 1,$ we can compute
  \begin{align*}
    v(\eta)
    &= v(\eta\wedge 0) + v(\eta\vee 0)\\
    &= -v(-(\eta\wedge 0)) + v(\eta\vee 0)\\
    &\leq -u(-(\eta\wedge 0)) + w(\eta\vee 0)\\
    &= u(\eta\wedge 0) + w(\eta\wedge 0)\\
    &\leq \norm{u}_\Hd + \norm{w}_\Hd,
  \end{align*}
  where for the last step we note that both $\eta\wedge 0$ and $\eta\vee 0$
  are $\Hsob$ functions with norm less than $\eta$ (see \eg \cite[Corollary
  2.1.8]{Ziemer}).
  \end{proof}

For the approximate deterministic dynamics governed by \eqref{eq:16}, we
then have the following comparison principle:
\begin{lem}\label{comp-principle}
  Let $x,y \in L^\infty \subseteq L^2$ and
  $x \leq y$ almost everywhere, and let $u^\xeps$ and $u^{y,\eps}$ be the solutions to
  \eqref{eq:16}
  with the corresponding initial values. Then
  \begin{displaymath}
    u^\xeps_t \leq u^{y,\eps}_t \quad \text{in }\Hd,\text{ for all }t>0.
  \end{displaymath}
\end{lem}
\begin{proof}
  Note that $u^\xeps$ for $x\in L^\infty$ is also a weak solution in the sense of \cite[Chapter 5]{VazquezPME} with
  $\Phi = \eps \mathrm{Id} + \phi^\eps$. By \cite[Theorem
  5.7]{VazquezPME}, the claimed comparison principle holds.
\end{proof}

\begin{Cor} \label{Cor-order}
  Let $R>0$. As a consequence of Lemmas \ref{sandwich-lemma} and \ref{comp-principle},
  we have for $x\in L^\infty$, $\norm{x}_\infty \leq R$ and arbitrary $u\in\Hd$
  \begin{displaymath}
    \norm{u^\xeps_t - u}_\Hd \leq \norm{u^{R,\eps}_t-u}_\Hd + \norm{u^{-R,
        \eps}_t - u}_\Hd\quad \text{for }t\geq 0.
  \end{displaymath}
\end{Cor}

\begin{proof}
  It is enough to read off Definition \ref{Def-Hd-order} that $-R\leq x \leq
  R$ almost everywhere implies $-R\leq x \leq R$ in $\Hd$, and that the
  order is invariant under translation by a fixed element of $\Hd$.
\end{proof}

We now compare the approximations $u^\xeps$ to the solution
of the stochastic equation \eqref{eq:34}, with a noise conditioned on
suitable events. 

\begin{lem}\label{stability-approx}
  Let $R, S>0, 0<\beta\leq 1, x\in L^\infty, \norm{x}_\infty
  \leq R$ and let $u^\xeps$ be the solution to \eqref{eq:16}. Furthermore, let
  $X^\xeps$ be the solution to \eqref{eq:34} up to time $S$ with the
  same initial condition $x$. Assume that
  \begin{equation}\label{eq:17}
    \sup_{t\in[0,S] } \norm{W_t^B - t\g}_2\leq \beta,
  \end{equation}
  where for simplicity we write $W^B_t = B W_t$. Then for $0<\eps\leq 1$ we have
  \begin{displaymath}
    \norm{X_S^\xeps - u_S^\xeps}_\Hd\leq C(R,S)\beta.
  \end{displaymath}
\end{lem}
\begin{proof}
  We consider the transformed processes
  \begin{align*}
    Y^\xeps_t &= X_t^\xeps - W^B_t\quad \text{and}\\
    v^\xeps_t &= u_t^\xeps - t\g,
  \end{align*}
  so that by
  \begin{displaymath}
    \norm{X_S^{x,\eps} - u_S^\xeps}_\Hd \leq \norm{Y^\xeps_S - v^\xeps_S}_\Hd +
    \norm{W^B_S - S\g}_\Hd,
  \end{displaymath}
  we can focus on $\norm{Y^\xeps_S - v^\xeps_S}_\Hd^2$ using \eqref{eq:17}
  and the continuity of the embedding $L^2\hookrightarrow\Hd$. For the
  following equalities, recall that $X^\xeps \in \Hsob$
  $\P\otimes \d t$--almost surely due to \eqref{eq:38}, and note that
  \begin{equation}\label{eq:41}
    \eps u_r^\xeps + \phi^\eps(u_r^\xeps) \in \Hsob 
  \end{equation}
  for almost every $r\in [0,S]$ by \cite[Theorem 5.7]{VazquezPME}. Hence
  \begin{align*}
    &\frac1{2}\norm{Y^\xeps_S - v^\xeps_S}_\Hd^2\\
    &= \int_0^S \sp{Y^\xeps_r - v^\xeps_r}{\Delta\left(\eps X_r^\xeps +
      \phi^\eps(X_r^\xeps)\right) - \Delta\left( \eps u^\xeps_r +
      \phi^\eps(u^\xeps_r)\right)}_\Hd \d r\\
    &= - \int_0^S \sp{Y^\xeps_r - v^\xeps_r}{\eps X_r^\xeps +
      \phi^\eps(X_r^\xeps) - \left(\eps u^\xeps_r + \phi^\eps(u^\xeps_r)\right)}_{L^2} \d r\\
    &= - \int_0^S \sp{Y^\xeps_r + W_r^B - (v^\xeps_r + r\g)}{\eps (Y^\xeps_r + W_r^B - (v^\xeps_r + r\g))}_{L^2} \d r\\
    &\quad - \int_0^S \sp{Y^\xeps_r + W_r^B - (v^\xeps_r+r\g)}{\phi^\eps(Y^\xeps_r
      + W^B_r) - \phi^\eps(v^\xeps_r + r\g)}_{L^2} \d r\\
    &\quad + \int_0^S \sp{W_r^B-r\g}{\eps (Y^\xeps_r + W_r^B - (v^\xeps_r + r\g)) +
      \phi^\eps(Y^\xeps_r + W^B_r) - \phi^\eps(v^\xeps_r + r\g)}_{L^2} \d r\\
    &\leq \int_0^S \norm{W_r^B-r\g}_2 \norm{\eps (Y^\xeps_r + W_r^B) + \phi^\eps(Y^\xeps_r
      + W^B_r) - \eps(v^\xeps_r + r\g) - \phi^\eps(v^\xeps_r + r\g)}_2 \d r\\
    &\leq \left(\int_0^S \norm{W_r^B-r\g}_2^2\d r\right)^{\frac1{2}}\\
    &\qquad \times \left(\int_0^S \left(\eps\norm{Y^\xeps_r + W_r^B}_2 + \norm{\phi^\eps(Y^\xeps_r
      + W^B_r)}_2 + \eps\norm{v^\xeps_r + r\g}_2 + \norm{\phi^\eps(v^\xeps_r +
      r\g)}_2\right)^2\d r\right)^{\frac1{2}}\\
    &\leq S^{\frac1{2}}\beta \left(4 \int_0^S \eps^2\norm{Y^\xeps_r + W_r^B}^2_2 + \norm{\phi^\eps(Y^\xeps_r
      + W^B_r)}^2_2 + \eps^2\norm{v^\xeps_r + r\g}_2^2 + \norm{\phi^\eps(v^\xeps_r + r\g)}_2^2\d r\right)^{\frac1{2}}.
  \end{align*}
  Note that the monotonicity of $\phi^\eps$ has been used for the first
  inequality. Provided that the last factor can be bounded for $R$ and $S$ fixed,
  letting $\beta \to 0$ can make the starting term arbitrarily small, as
  required.

  To see this boundedness, first notice by \eqref{eq:40} in Appendix
  \ref{Specific-Yosida} that $\abs{\phi^\eps(x)}\leq \abs{x}$ for all
  $x\in\R$, $\eps>0$, so that it is enough to prove suitable bounds on
  \begin{displaymath}
    \int_0^S\norm{Y^\xeps_r + W_r^B}^2_2\,\d r \quad\text{and}\quad
    \int_0^S\norm{v^\xeps_r + r\g}_2^2\,\d r.
  \end{displaymath}
  To this end, we can compute
  \begin{equation}\label{eq:42}
    \frac1{2}\norm{Y^\xeps_S}_\Hd^2
    = \norm{x}_\Hd^2 + \int_0^S \sp{\eps\Delta(X^\xeps_r) +
       \Delta\phi^\eps(X^\xeps_r)}{Y^\xeps_r}_\Hd\d r
   \end{equation}
   by \eqref{eq:38}, and further, noting $Y^\xeps_r\in L^2$ by \eqref{eq:38} and
   \eqref{eq:21},
   \begin{align}\label{eq:43}
     \begin{split}
       \eqref{eq:42}
       =&\norm{x}_\Hd^2 -  \int_0^S \sp{\eps X^\xeps_r +
         \phi^\eps(X^\xeps_r)}{Y^\xeps_r}_{L^2}\d r\\
       =& \norm{x}_\Hd^2 - \int_0^S \sp{\eps (Y^\xeps_r + W^B_r) +
         \phi^\eps(Y^\xeps_r + W^B_r)}{Y^\xeps_r + W^B_r}_{L^2}\d r\\
       &+ \int_0^S \sp{\eps (Y^\xeps_r + W^B_r) +
         \phi^\eps(Y^\xeps_r + W^B_r)}{W^B_r}_{L^2}\d r.
     \end{split}
   \end{align}
   From \eqref{eq:40} in Appendix \ref{Specific-Yosida}, we obtain the
   lower bound $\abs{\phi^\eps(x)} \geq \frac1{2}\abs{x}$ for $\abs{x} \geq 1+\eps$ and
   $\eps \leq 1$, so that for $u\in L^2$ we have the estimate
   \begin{equation}\label{eq:18}
    \norm{u}_2^2
    \leq \underbrace{\int_{\{\abs{u}\geq 1+\eps\}} 2u \phi^\eps(u)\,\d
      x}_{\leq 2\sp{u}{\phi^\eps(u)}_{L^2}} +
    4\abs{\cO}
    \leq \underbrace{\int_{\{\abs{u}\geq 1+\eps\}}4\phi^\eps(u)^2\,\d
      x}_{\leq 4\norm{\phi^\eps(u)}_2^2} +
    4\abs{\cO}.
  \end{equation}

   Using \eqref{eq:18} and Young's inequality for the last two summands,
   once weighted by $\frac1{2}$, we can continue by
   \begin{align}\label{eq:44}
     \begin{split}
     \eqref{eq:43}
     \leq& \norm{x}_\Hd^2 - \int_0^S \eps \norm{Y^\xeps_r+W^B_r}_2^2 +
     \frac1{2}\norm{Y^\xeps_r+W^B_r}_2^2 - C\,\d r\\
     &+ \int_0^S \frac{\eps}{2} \norm{Y^\xeps_r + W^B_r}_2^2 +
     \frac{\eps}{2} \norm{W^B_r}_2^2 +
     \frac1{4}\underbrace{\norm{\phi^\eps(Y^\xeps_r + W^B_r)}_2^2}_{\leq
       \norm{Y^\xeps_r+W^B_r}_2^2} +
     \norm{W^B_r}_2^2\,\d r\\
     \leq& \norm{x}_\Hd^2 - \frac1{4}\int_0^S \norm{Y^\xeps_r +
       W^B_r}_2^2\d r + \frac{3}{2}\int_0^S \norm{W^B_r}_2^2 + C\,\d r.
   \end{split}
  \end{align}
  We note that by \eqref{eq:17}, assumption \eqref{eq:21} and $\beta\leq 1$
  \begin{displaymath}
    \norm{W_r^B}_2 \leq \norm{W_r^B - r\g}_2 + \norm{r\g}_2 \leq \beta +
    S\norm{\g}_2 \leq C(S),
  \end{displaymath}
  such that \eqref{eq:44} yields, by dropping the left-hand side and
  relabelling the constants,
  \begin{equation}\label{eq:45}
    \int_0^S\norm{Y^\xeps_r + W_r^B}^2_2\,\d r \leq 4 C(S, \norm{x}_\Hd^2).
  \end{equation}
  To obtain a bound that only depends on $S$ and $R$, note that $x\in L^\infty,
  \norm{x}_\infty \leq R$ by assumption, such that
  \begin{displaymath}
    \norm{x}_\Hd \leq C\norm{x}_2 \leq 2\, C\abs{\cO}^{\frac1{2}}R,
  \end{displaymath}
  which, together with \eqref{eq:45}, yields the desired bound. A
  similar estimate for $\int_0^S\norm{v^\xeps_r + r\g}_2^2\,\d r$ can be
  obtained by analogous computations, using \eqref{eq:41} instead of
  \eqref{eq:38}.
\end{proof}

We need to ensure that \eqref{eq:17} is realised for each $\beta>0$ with
non-zero probability.
\begin{lem}\label{Support-prop}
  As in \eqref{eq:25} we denote
  \begin{displaymath}
    W^B_t = B W_t = \sum_{i=1}^\infty \beta_k(t) \xi_k,
  \end{displaymath}
  with $\sum_{k\in\N} \norm{\xi_k}_2^2 < \infty$. Let $\g$ be defined as in
  \eqref{eq:21}, and let the degeneracy assumption on $(\xi_k)$ in \eqref{eq:21}
  hold. Then for all
  $S\geq 0, \beta > 0$ we have
  \begin{displaymath}
    \P\left(\sup_{t\in[0,S]} \norm{W_t^B - t\g}^2_2\leq \beta \right)>0.
  \end{displaymath}
\end{lem}
\begin{proof}
  We use orthogonality of $(\xi_k)_k$ to  write, for $m^* > m$,
  \begin{align}\label{eq:28}
    \begin{split}
      \norm{W_t^B - t \g}^2_2
      &= \norm{\sum_{k=1}^m \xi_k(\beta_k(t) - t c_k)}^2_2
      + \norm{\sum_{k=m+1}^{m^*} \xi_k \beta_k(t)}^2_2 + \norm{\sum_{k=
          m^*+1}^\infty \xi_k \beta_k(t)}^2_2\\
      &= \sum_{k=1}^m \norm{\xi_k}_2^2\abs{\beta_k(t) - t c_k}^2
      + \sum_{k=m+1}^{m^*} \norm{\xi_k}^2_2 \abs{\beta_k(t)}^2 + \sum_{k= m^*+1}^\infty \norm{\xi_k}^2_2 \abs{\beta_k(t)}^2.
    \end{split}
  \end{align}
  For the first term, we note that the event
  \begin{equation}\label{eq:27}
    \max_{k\in\{1,\dots,m\}}\sup_{t\in[0,S]}\abs{\beta_k(t) - c_kt}^2
    \leq \frac{\beta}{3\sum_{k=1}^m \norm{\xi_k}_2},
  \end{equation}
  has positive probability by the following reasoning: As the $(\beta_k)_{k=1}^m$
  are independent, it is enough to show for a one-dimensional standard
  Brownian motion ($k\in\{1,\dots,m\}$) that
  \begin{equation}\label{eq:24}
    \P\left(\sup_{t\in[0,S]} \abs{\beta_k(t) - c_kt} \leq \eps\right) > 0
  \end{equation}
  for any fixed $S>0, \eps>0$. To see this, note that $\beta_k(t) -c_kt$ is
  again a standard Brownian motion with respect to some probability measure
  $\P_Q$, which is absolutely continuous to $\P$ by Girsanov's
  theorem. Thus, it is enough to show for a standard Brownian motion
  $\beta_1$ that
  \begin{equation}\label{eq:26}
    \P\left(\sup_{t\in[0,S]}\abs{\beta_1(t)} \leq \epsilon\right) > 0,
  \end{equation}
  as this is equivalent to
  \begin{displaymath}
    \P_Q\left(\sup_{t\in[0,S]} \abs{\beta_k(t) - c_kt} \leq \eps\right) >
    0,
  \end{displaymath}
  which by absolute continuity yields
  \begin{displaymath}
       \P\left(\sup_{t\in[0,S]} \abs{\beta_k(t) - c_kt} \leq
      \eps\right) > 0.
  \end{displaymath}
  In order to show \eqref{eq:26}, we first note that the exit time probability
  $\P(\sup_{t\in[0,S]}\abs{\beta_1(t)} < \eps) = p(0, S)$ solves a
  Kolmogorov backward equation on $[-\eps, \eps]\times[0,\infty)$, as shown
  \eg in \cite{Patie-Winter}, which reads
  \begin{eqnarray*}
    \partial_t p(x,t) &=& \frac1{2}\Delta p(x,t), \quad \text{on }(-\eps,\eps)\times(0,\infty),\\
    p(x,t=0) &\equiv& 1,\text{for }x\in(-\eps,\eps)\\
    p(x= \pm \eps,t) &=& 0,\text{for }t>0.
  \end{eqnarray*}
  By the strong maximum principle, we can conclude that $p(0,S)>0$ for
  arbitrary $S>0$. Thus, we have shown that \eqref{eq:27} has positive
  probability and thus
  \begin{displaymath}
    \P\left(\sum_{k=1}^m \norm{\xi_k}_2^2\abs{\beta_k - t c_k}^2 >
      \frac{\beta}{3}\right) > 0.
  \end{displaymath}

  For the third term in \eqref{eq:28}, we compute
  \begin{displaymath}
    \E\sup_{t\in[0,S]} \sum_{k>m^*} \abs{\beta_k(t)}^2\norm{\xi_k}_2^2
    \leq \sum_{k>m^*} \norm{\xi_k}_2^2 \E\sup_{t\in[0,S]}
    \abs{\beta_k(t)}^2 \leq 4S \sum_{k>m^*} \norm{\xi_k}_2^2 =: R(m^*) \searrow 0
  \end{displaymath}
  for $m^* \to \infty$, where we used the squared version of the Burkholder-Davis-Gundy
  inequality. Choosing $m^*$ so large that $R(m^*) \leq \frac{\beta}{3}$ we
  obtain
  \begin{displaymath}
    \P\left(\sup_{t\in[0,S]}\sum_{k>m^*}\norm{\xi_k}_2^2\abs{\beta_k(t)}^2\leq
      \frac{\beta}{3}\right) \geq 1
    - \frac{R(m^*)}{\frac{\beta}{3}} > 0.
  \end{displaymath}
  Having chosen $m^*$ in this way, we can now conclude by \eqref{eq:26} that
  also for the second term of \eqref{eq:28} we have
  \begin{displaymath}
    \P\left(\sum_{k=m+1}^{m^*} \norm{\xi_k}^2_2 \abs{\beta_k}^2 \leq
      \frac{\beta}{3}\right) > 0,
  \end{displaymath}
  which proves the claim by independence.
\end{proof}

Combining the results up to now, we can state:
\begin{lem}\label{conv-uinfty}
  Let $\delta>0, R > 1$ and let $g\in L^2$ satisfy assumption
  \eqref{eq:21}. Recall $u_\infty$ from Lemma \ref{det-conv-applied} as the
  long-time limit of solutions $u^{R}$, $u^{-R}$ to \eqref{eq:13}. Then
  there exist $\gamma, S>0$ such that for every initial value
  $x \in C_\delta(R)$, where $C_\delta(R)$ is the $\delta$-neighbourhood of
  $B_R^\infty(0)$ in $\Hd$, we have
  \begin{displaymath}
    \P(\norm{X^x_S - u_\infty}_\Hd < 2\delta) \geq \gamma.
  \end{displaymath}
\end{lem}
\begin{proof}
  Recall that $u^R$, $u^{-R}$ are well-defined by Remark
  \ref{Well-posedn-Brez} and Lemma \ref{det-conv-applied}. According to
  Lemma \ref{det-conv-applied}, we can choose $S>0$ such that we have
  \begin{equation}\label{eq:15}
    \max\left\{\norm{u^{R}(t) - u_\infty}_\Hd, \norm{u^{-R}(t) - u_\infty}_\Hd\right\} \leq \frac{\delta}{8}\quad \text{for all
    }t\geq S.
  \end{equation}
  Let $u^\xeps$ be defined as in Lemma \ref{stability-approx}. As shown in
  this Lemma, we can choose $0<\beta\leq 1$ such that
  \begin{equation}
    \sup_{t\in[0,S]} \norm{W_t^B - t\g}_2 \leq \beta\quad\text{implies }
    \norm{X_S^\xeps - u_S^\xeps}_\Hd < \frac{\delta}{4},
  \end{equation}
  uniformly for all $\eps\in (0, 1]$, $x\in B_R^\infty(0)$. We then define
  \begin{equation}\label{eq:20}
    \gamma := \frac{2}{3}\ \Plr{\sup_{t\in[0,S]} \norm{W_t^B - t\g}^2_2\leq \beta},
  \end{equation}
  which is strictly positive by Lemma \ref{Support-prop}. We then choose
  $\eps\in(0,1]$ small enough such that for $u^{R,\eps}$ and $u^{-R,\eps}$ as in \eqref{eq:16} we have
  \begin{equation}\label{eq:32}
    \max\left\{\norm{u^{R, \eps}_S - u^{R}_S}_\Hd, \norm{u^{-R, \eps}_S - u^{-R}_S}_\Hd\right\} \leq \frac{\delta}{8},
  \end{equation}
  which is possible by \eqref{eq:58}, and such that
  \begin{equation}\label{eq:31}
    \E\sup_{r\in[0,S]}\norm{X^\xeps_r- X_r^x}_\Hd \leq \frac{\gamma\delta}{8}
  \end{equation}
  holds uniformly for $x\in B_R^\infty(0)$ by \eqref{eq:47} (note that the squared
  form in \eqref{eq:47} is a stronger statement than needed for \eqref{eq:31} by
  Jensen's inequality). For every $x\in B_R^\infty(0)$, this leads to
  \begin{equation}
    \Plr{\norm{X^x_S - X_S^\xeps}_\Hd \leq \frac{\delta}{4}} \geq 1 - \frac{\gamma}{2},
  \end{equation}
  and, by Corollary \ref{Cor-order}, \eqref{eq:15} and \eqref{eq:32}, to
  \begin{align*}
    \norm{u^\xeps_S - u_\infty}_\Hd
    &\leq \norm{u^{R,\eps}_S-u_\infty}_\Hd + \norm{u^{-R,\eps}_S -
      u_\infty}_\Hd\\
    &\leq \norm{u^{R,\eps}_S -u^R_S}_\Hd + \norm{u^R_S - u_\infty}_\Hd +
      \norm{u^{-R,\eps}_S - u^{-R}_S}_\Hd + \norm{u^{-R}_S - u_\infty}_\Hd\\
    &\leq 4\, \frac{\delta}{8} = \frac{\delta}{2}.
  \end{align*}
  Hence, still for $x\in B_R^\infty(0)$, we can conclude,
  \begin{align*}
    \Plr{\norm{X^x_S - u_\infty}_\Hd< \delta}
    &\geq \Plr{\norm{X^x_S - X_S^\xeps}_\Hd < \frac{\delta}{4}\text{ and }
      \norm{X_S^\xeps - u_S^\xeps}_\Hd < \frac{\delta}{4}}\\
    &= 1 - \Plr{\norm{X^x_S - X_S^\xeps}_\Hd \geq \frac{\delta}{4} \text{ or }
      \norm{X_S^\xeps - u_S^\xeps}_\Hd \geq \frac{\delta}{4}}\\
    &\geq 1 - \Plr{\norm{X^x_S - X_S^\xeps}_\Hd \geq \frac{\delta}{4}} -
      \Plr{\norm{X_S^\xeps - u_S^\xeps}_\Hd \geq \frac{\delta}{4}}\\
    &\geq \Plr{\norm{X_S^\xeps -u_S^\xeps}_\Hd < \frac{\delta}{4}} -
      \frac{\gamma}{2}\\
    &\geq \P\left(\sup_{t\in[0,S]} \norm{W_t^B - t\g}^2_2\leq \beta \right)
      - \frac{\gamma}{2} = \gamma.
  \end{align*}
  The claim for $x\in C_\delta(R)$ follows immediately by Lemma \ref{Contractivity-prop}.
\end{proof}
\begin{proof}[Proof of Theorem \ref{main-result}]
  Lemma \ref{Markov-prop}, Remark \ref{existence-canonical-process} and
  Remark \ref{stoch-continuity} prove all requirements of Proposition
  \ref{KPS} except \eqref{eq:56}. To see this remaining statement, we
  estimate for $0<\rho<1$ and $R(\rho)$ given in Lemma
  \ref{lem-energy-estimate}
  \begin{align*}
    &\liminf_{T\to\infty} \frac1{T} \int_0^T P_r\left(x,B_{2\delta}^\Hd(u_\infty)\right)\d r\\
    &= \liminf_{T\to\infty} \frac1{T} \int_0^T P_{r+S}\left(x,B_{2\delta}^\Hd(u_\infty)\right)\d r\\
    &= \liminf_{T\to\infty} \frac1{T} \int_0^T \int_\Hd P_S\left(y,
      B_{2\delta}^\Hd(u_\infty)\right)\,P_r(x,\d y)\, \d r\\
    &\geq \liminf_{T\to\infty} \frac1{T} \int_0^T \int_{C_\delta(R(\rho))}
      \underbrace{P_S\left(y, B_{2\delta}^\Hd(u_\infty)\right)}_{\geq \gamma\text{
      by Lemma \ref{conv-uinfty}}}\,P_r(x,\d y)\, \d r\\
    &\geq \gamma \liminf_{T\to\infty} \frac1{T} \int_0^T P_r(x,
      C_\delta(R(\rho)))\,\d r > \gamma\,(1- \rho) > 0,
  \end{align*}
  where we used the semigroup property of $(P_t)_{t\geq 0}$ and, for the
  last step, Lemma \ref{lem-energy-estimate}. The result then follows by
  Theorem \ref{KPS}.
\end{proof}

\appendix

\section{Solutions to monotone evolution equations}
For the reader's convenience, we cite and translate \cite[Definition 3.1]{Brezis}:
\begin{Def}\label{Def-det-solns}
  Let $H$ be a Hilbert space, $f\in L^1([0,T]; H)$,
  $A: H\supseteq D(A) \to H$ a maximal monotone operator. A function
  $u\in C([0,T]; \Hd)$ is called a \textbf{strong solution} to
  \begin{equation}\label{eq:36}
    \frac{\d}{\d t}  u \in -Au + f,
  \end{equation}
  if $u$ is absolutely continuous on compact subsets of $(0,T)$ (which
  implies that $u$ is differentiable almost everywhere in $(0,T)$) and for
  almost all $t\in (0,T)$
  \begin{displaymath}
    u(t) \in D(A)
  \end{displaymath}
  and
  \begin{displaymath}
    \frac{\d u}{\d t}(t) \in -Au(t) + f(t).
  \end{displaymath}

  $u\in C([0,T]; \Hd)$ is called a \textbf{weak solution} to \eqref{eq:36}
  if there are sequences $f_n\in L^1([0,T]; H)$ and $u_n \in C([0,T]; H)$
  $(n\in\N)$ such that $u_n$ is a strong solution of the equation
  \begin{displaymath}
    \frac{\d}{\d t} u_n \in -Au_n + f_n,
  \end{displaymath}
  $f_n\to f$ in $L^1([0,T]; H)$ and $u_n\to u$ uniformly
  in $[0,T]$ for $n\to \infty$.
\end{Def}

\begin{Rem}
  We observe that each strong solution is also a weak solution.
\end{Rem}

\section{Yosida approximation for the specific function $\phi$}\label{Specific-Yosida}
Recall from section \ref{sec:setting} that the multivalued function
$\phi: \R\to\R$ is defined as the maximal monotone extension of
\begin{displaymath}
  \R\ni x \mapsto x \Ind{\{\abs{x}>1\}}.
\end{displaymath}
We want to explicitly calculate its resolvent function $R^\eps: \R\to\R$
and its Yosida approximation $\phi^\eps:
\R\to\R$. For theoretical details, see \cite[Appendix D]{Neuss}. 

The resolvent $R^\eps(x)$ is defined as the solution $s$ to
\begin{equation}\label{eq:39}
  s + \eps \phi(s) \ni x.
\end{equation}
Note that \eqref{eq:39} has exactly one solution by the maximal
monotonicity of $\phi$. For $x\in [-1,1]$ we have
\begin{displaymath}
  0\in \phi(x),
\end{displaymath}
thus \eqref{eq:39} is solved by $s=x$. Consequently $R^\eps(x) = x$.

For $x\in(1, 1+\eps]$ we have
\begin{displaymath}
  \frac{x-1}{\eps} \in [0,1] = \phi(1).
\end{displaymath}
Thus, $s=1$ solves the equation by
\begin{displaymath}
  x = 1 + \eps\frac{x-1}{\eps} \in 1 + \eps\phi(1),
\end{displaymath}
which yields $R^\eps(x) = 1$. If $x\in [-1-\eps, 1)$, the same argument
yields $R^\eps(x) = -1$.

For $\abs{x}>1+\eps$, we have $\abs{\frac{x}{1+\eps}} > 1$ such that
\begin{displaymath}
  \frac{x}{1+\eps} + \eps\phi\left(\frac{x}{1+\eps}\right) = \frac{x}{1+\eps} + \eps
  \frac{x}{1+\eps} = x,
\end{displaymath}
yielding $R^\eps(x) = \frac{x}{1+\eps}$. By definition of the Yosida
approximation,
\begin{displaymath}
  \phi^\eps(x) = \frac{x-R^\eps(x)}{\eps},
\end{displaymath}
it is now easy to conclude that
\begin{equation}\label{eq:62}
  \phi^\eps(x) =
  \begin{cases}
    0, \quad &\abs{x}\leq 1\\
    \frac{x-1}{\eps} &x\in (1, 1+\eps]\\
    \frac{x+1}{\eps} &x\in [-1-\eps, 1)\\
    \frac{x}{1+\eps} &\abs{x}>1+\eps
  \end{cases}.
\end{equation}
In particular, for $\eps\leq1$ and $\abs{x} \geq 1+\eps$, we observe that
\begin{equation}\label{eq:40}
  \abs{\phi^\eps(x)} \geq \frac{\abs{x}}{2}.
\end{equation}

\pdfbookmark{References}{references}

\bibliographystyle{abbrv}
\bibliography{mybooks}

\begin{flushleft}
\small \normalfont
\textsc{Marius Neuß\\
Max--Planck--Institut f\"ur Mathematik in den Naturwissenschaften\\
04103 Leipzig, Germany}\\
\texttt{\textbf{marius.neuss@mis.mpg.de}}
\end{flushleft}

\end{document}